\documentclass[12pt,reqno]{amsart}
\usepackage{fullpage}
\usepackage{amssymb,amsmath,amsthm,latexsym}
\usepackage[usenames]{color}
\usepackage{graphicx}
\usepackage[colorlinks=true,
linkcolor=webgreen,
filecolor=webbrown,
citecolor=webgreen]{hyperref}

\definecolor{webgreen}{rgb}{0,.5,0}
\definecolor{webbrown}{rgb}{.6,0,0}
\definecolor{red}{rgb}{1,0,0}

\usepackage[T1]{fontenc}
\newtheorem{theorem}{Theorem}[section]

\newtheorem{lemma} [theorem]{Lemma}

\newtheorem{remark}[theorem]{Remark}

\usepackage{enumerate}
\usepackage{color}

\begin{document}
\title[On transcendence of non-periodic continued fractions associated with modular forms and arithmetic functions]{On transcendence of non-periodic continued fractions associated with modular forms and arithmetic functions}

\author[Tapas Chatterjee]{Tapas Chatterjee}

\address{%
Department of Mathematics,\\Indian Institute of Technology Ropar, Punjab, India.}

\email{tapasc@iitrpr.ac.in}

\author{Sagar Mandal$^{*}$}
\address{Department of Mathematics,\\Indian Institute of Technology Ropar, Punjab, India.}
\email{sagar.25maz0008@iitrpr.ac.in, sagarmandal31415@gmail.com}
\thanks{$^{*}$The second author is financially supported by the University Grants Commission (UGC), Government of India, through the award of the Junior Research Fellowship (JRF) with ref. No. 241620111598.}

	
	
	
\maketitle
    
    \begin{abstract}
The purpose of this article is two-folds. Firstly, we establish two sufficient conditions under which the sequence $\{f(n)\pmod{m};n\geq1\}$ is non-periodic, where $f(n)$ is an arithmetic function. As consequences, we deduce that the sequences associated with the Ramanujan tau function $\tau(n)$ as well as the Fourier coefficients of certain normalized Eisenstein series $E_k(z)$ modulo $m$ are non-periodic. Further, we deduce that the sequence arising from Nathanson’s totient function $\Phi(n)$, the classical Euler's totient function $\varphi(n)$, sum of divisor function $\sigma(n)$, their Dirichlet convolution $\sigma * \varphi(n)$, Jordan's totient function $J_k(n)$, and unitary totient function $\varphi^*(n)$ modulo  $m$, are non-periodic for certain modulo $m$. In addition, we extend a result of Ayad and Kihel \cite{r1} on the non-periodicity of certain arithmetic function $g(n)$. On the other hand, we construct several transcendental numbers arising from the continued fractions attached with $\tau(n)$, $E_k(z)$, $\Phi(n)$, $g(n)$, $\varphi(n)$, $\sigma(n)$, $\sigma * \varphi(n)$, $J_k(n)$, and $\varphi^*(n)$. 
\end{abstract}

\tableofcontents
\section{Introduction}
A periodic sequence with period $L$ is a sequence $a_1,a_2,a_3\ldots$ satisfying $a_{n+L}=a_{n}$ for all $n\in \mathbb{N}$. For any arithmetic function $f(n)$, we say that the sequence $\{f(n)\pmod{m}\}_{n\geq1}$ is periodic if there exists a fixed positive integer $L$ such that $f(n+L)\equiv f(n) \pmod{m}\quad \text{for all~}n\in \mathbb{N}$. Also, $\{f(n) \pmod{m}\}_{n\geq 1}$ is called an eventually periodic sequence if there exist fixed positive integers $N,L$ such that $f(n+L)\equiv f(n) \pmod{m}\quad \text{for all~}n\geq N$. It is called a non-periodic sequence if it is not eventually periodic. 

Our first objective in this paper is to develop general criteria that guarantee the non-periodicity modulo $m$ of sequences arising from arithmetic functions. These criteria are then applied to a broad family of classical functions coming from modular forms and multiplicative number theory.

 Let $k$ be a positive integer. The space of modular forms of weight $k$ for the full modular group is denoted by $M_k(\mathrm{SL}_2(\mathbb{Z}))$. Every modular form $f \in M_k(\mathrm{SL}_2(\mathbb{Z}))$ admits a Fourier expansion at the cusp $i\infty$,
$$f(z) = \sum_{n=0}^{\infty} a_n q^n, \qquad q = e^{2\pi i z}.$$
The coefficient $a_0$ is called the constant term of $f$. A modular form is
a cusp form precisely when this constant term vanishes, or equivalently, $\lim_{z\to i\infty} f(z)=0$.
The space of cusp forms is denoted by $S_k(\mathrm{SL}_2(\mathbb{Z}))$. A nontrivial example of a modular form of weight $k$ for $\mathrm{SL}_2(\mathbb{Z})$ is the Eisenstein series, for $k \geq 4$, the Eisenstein series of weight $k$ is defined by 
$$G_k(z):= \sum_{(m,n)\in \mathbb{Z}^2 \setminus \{(0,0)\}}
    \frac{1}{(mz+n)^k}.$$
For every $k \ge 4$, the normalized Eisenstein series of weight $k$ is defined by
$$E_k(z):= \frac{G_k(z)}{2\zeta(k)}=1-\frac{2k}{B_k}\sum_{n=1}^{\infty} \sigma_{k-1}(n)q^n,$$
note that $E_k(z)$ is also a modular form of weight $k$ for $\mathrm{SL}_2(\mathbb{Z})$, where $B_k, \zeta(k),\sigma_{k}(n)$ are Bernoulli numbers, the Riemann zeta function, and generalized sum-of-divisors function respectively. For further details on modular forms, see \cite{murty}. The Ramanujan tau-function \cite{new1,new2} (\href{https://oeis.org/A000594}{A000594}) $\tau(n)$, which gives the coefficients of the unique normalized weight 12 cusp form for $\mathrm{SL}_2(\mathbb{Z})$ (with $q := e^{2\pi i z}$), is defined by
$$\Delta(z) = \sum_{n=1}^{\infty} \tau(n) q^n
= q \prod_{n=1}^{\infty} (1-q^n)^{24}
= q-24q^2+252q^3-1472q^4+4830q^5-\cdots.$$
This function has played a central role in the theory of modular forms. The multiplicativity property of the tau function $\tau(n)$ plays a pivotal role in the subsequent proofs, the multiplicativity property of $\tau(n)$ was originally conjectured by Ramanujan, which was later established by Mordell \cite{Mordell}.\\
We prove that the sequences associated with the Ramanujan tau function $\tau(n)$ as well as the Fourier coefficients of certain normalized Eisenstein series $E_k(z)$ modulo $m$ are non-periodic. The following theorems proceed in this direction.

\begin{theorem}\label{theorem tau}
$\{\tau(n) \pmod{m}\}_{n\geq1}$ is a non-periodic sequence for $m=5,7,8,9,691$.
\end{theorem}
\begin{remark}
Using the same techniques used in this paper, one can obtain the non-periodicity results for $\{\tau(n) \pmod{m}\}_{n\geq1}$ for other $m$. For our purposes related to transcendence, we establish the result only for $m=5,7,8,9$. Moreover, since one of the most celebrated congruences involving $\tau(n)$ is the congruence modulo $691$, we also include the case $m=691$.
\end{remark}

\begin{theorem}\label{theorem es} We have
\noindent
\begin{enumerate}
\item[\upshape(1)] If $E_4(z)= \sum_{n\geq0}a(n)q^n$ then the sequence $\{a(n) \pmod{m}\}_{n\geq1}$ is non-periodic, where $m\geq3$ with $\gcd(m,240)=1$.
\item[\upshape(2)] If $E_6(z)= \sum_{n\geq0}b(n)q^n$ then the sequence $\{b(n) \pmod{m}\}_{n\geq1}$ is non-periodic, where $m\geq3$ with $\gcd(m,504)=1$, and $m\neq11$.
\item[\upshape(3)] If $E_8(z)=\sum_{n\geq0}c(n)q^n$ then the sequence $\{c(n) \pmod{m}\}_{n\geq1}$ is non-periodic, where $m\geq3$ with $\gcd(m,480)=1$, and $m\neq 43$.
\item[\upshape(4)] If $E_{10}(z)=\sum_{n\geq0}d(n)q^n$ then the sequence $\{d(n) \pmod{m}\}_{n\geq1}$ is non-periodic, where $m\geq3$ with $\gcd(m,264)=1$, and $m\neq 19$.
\item[\upshape(5)] If $E_{14}(z)=\sum_{n\geq0}e(n)q^n$ then the sequence $\{e(n) \pmod{m}\}_{n\geq1}$ is non-periodic, where $m\geq3$ with $\gcd(m,24)=1$, and $m\neq 2731$.
\end{enumerate}
\end{theorem}
Sloane’s sequence \href{https://oeis.org/A004009}{A004009} enumerates $a(n)$,  \href{https://oeis.org/A013973}{A013973} enumerates $b(n)$, \href{https://oeis.org/A013974}{A013974} enumerates $d(n)$ and \href{https://oeis.org/A058550}{A058550} enumerates $e(n)$.

 In number theory, the concept of relatively prime subsets generalizes the idea of coprime integers. A nonempty finite set $A \subseteq \mathbb{N}$ is called relatively prime if $\gcd(A) = 1$, meaning that the $\gcd$ of elements of $A$ is $1$. This notion was first introduced by Nathanson \cite{1}. Let $g(n)$ denote the number of relatively prime subsets of $ \{1, 2, \dots, n\}$. Nathanson also defined analogues of Euler’s totient function $\Phi(n)$ which counts the number of nonempty sets $A \subseteq \{1, 2, \dots, n\}$ for which $\gcd(A)$ is relatively prime to $n$. These functions satisfy Möbius inversion formulas, which are the following: 
\begin{theorem}[\cite{1}, Theorem 1, Theorem 3] For all positive integers $n>1$, we have
    $$\Phi(n) = \sum_{d \mid n} \mu(d) \cdot 2^{n/d}, \quad g(n)=\sum_{d=1}^{n}\mu(d)\big(2^{\lfloor n/d\rfloor}-1\big),$$
    where $\mu$ is the Möbius function and $\lfloor x \rfloor$ denotes the greatest integer less than or equal to $x$. 
\end{theorem} 
\noindent
Sloane’s sequence \href{https://oeis.org/A085945}{A085945} enumerates $g(n)$ and \href{https://oeis.org/A027375}{A027375} enumerates $\Phi(n)$ for $n\geq2$. Jordan’s totient function, denoted by $J_k(n)$ (where $k$ is a positive integer), is a generalization of Euler’s totient function. Jordan's totient function $J_k$ is a multiplicative arithmetic function and may be evaluated as
$$J_k(1)=1,\quad J_k=n^k\prod_{p\mid n}\big(1-\frac{1}{p^k}\big)~~\text{for~~}n\geq2.$$
On the other hand, the unitary analogue of the Euler totient function $\varphi^{*}(n)$ is defined as
$$\varphi^*(1)=1,\quad \varphi^{*}(n)=\prod_{p^{\alpha}\mid\mid n}\big(p^{\alpha}-1\big)~~\text{for~~}n\geq2.$$
One can note that $\varphi^*(n)$ is also multiplicative.

Studying the sequences $\{\Phi(n)\}_{n\geq1}$ and $\{g(n)\}_{n\geq1}$ is not new, for example, Ayad and Kihel \cite{r1} proved the following theorem for the sequence $\{g(n) \pmod{m}\}_{n\geq1}$.
\begin{theorem}[\cite{r1}, Theorem 3]
$\{g(n) \pmod{m}\}_{n\geq1}$ is a non-periodic sequence for prime numbers $m\neq3$.
\end{theorem}
In the other direction, Bachraoui and Luca \cite{r3} proved that none of the sequences $\{\Phi(n)\}_{n\geq1}$ and $\{g(n)\}_{n\geq1}$ is P-recursive. 
We begin with generalising  Ayad and Kihel's problem of the non-periodicity of the sequence $\{g(n)\pmod{m}\}_{n\geq1}$. Here is the statement of our result:
\begin{theorem}\label{thm1.3}
If $g(n)$ denotes the number of relatively prime subsets of $\{1, 2, \dots, n\}$ then 
$\{g(n) \pmod{m}\}_{n\geq1}$ is a non-periodic sequence for all odd positive integers $m\geq 5$.
\end{theorem}

In addition, we prove the non-periodicity property for the sequence arising from Nathanson’s totient function $\Phi(n)$, as well as the classical Euler's totient function $\varphi(n)$, sum of divisor function $\sigma(n)$, their Dirichlet convolution $\sigma * \varphi(n)$, Jordan's totient function $J_k$, and unitary totient function modulo $m$. The following are our theorems in this direction:
\begin{theorem}\label{thm1.4}
$\{\Phi(n) \pmod{m}\}_{n\geq 1}$ is a non-periodic sequence for all odd positive integers $m\geq 5$.
\end{theorem}
\begin{theorem}\label{thm1.5} Under the above notations, we have the following
\noindent
\begin{enumerate}
    \item[\upshape(1)]  $\{\varphi(n) \pmod{m}\}_{n\geq 1}$ is a non-periodic sequence for all positive integers $m\geq 3$. 
    \item[\upshape(2)]  $\{\sigma(n) \pmod{m}\}_{n\geq 1}$ is a non-periodic sequence for all positive integers $m\geq 7$. 
    \item[\upshape(3)] $\{\sigma*\varphi(n) \pmod{m}\}_{n\geq 1}$ is a non-periodic sequence for all positive integers $m\geq 7$. 
    \item[\upshape(4)]  $\{J_k(n) \pmod{m}\}_{n\geq 1}$ is a non-periodic sequence for all positive integers $m\geq 3$ 
    and for all odd integers $k\geq3$.  
    \item[\upshape(5)]  $\{\varphi^*(n) \pmod{m}\}_{n\geq 1}$ is a non-periodic sequence for all positive integers $m\geq 3$.
\end{enumerate}
\end{theorem}

In order to prove the preceding theorems, we make use of the following general theorems.
\begin{theorem}\label{thm1.8}
 Let $f$ be an arithmetic function and $m$ be a positive integer such that $m\nmid f(p)$ for infinitely many primes $p$ and $m\mid f(p)$ for primes $p\equiv 1 \pmod{K(m)}$, where $K(m)$ is some positive integer-valued function of $m$. If $f$  has the property that $f(p)\mid f(n)$ whenever $p\mid n$, where $p$ is a prime, then $\{f(n) \pmod{m}\}_{n\geq1}$ is non-periodic.   
\end{theorem}
\begin{theorem}\label{thm1.9}
Let $f$ be an arithmetic function and $m$ be a positive integer such that $f(p)\equiv r_1\pmod{m}$ for primes $p\equiv1\pmod{m}$, $f(p)\equiv r_2\pmod{m}$ for infinitely many primes $p$ such that $r_2,r_2r_1$ are distinct in $\mathbb{Z}/m\mathbb{Z}$. If $f$ is multiplicative function for square-free integers $n$ then $\{f(n) \pmod{m}\}_{n\geq1}$ is non-periodic.  
\end{theorem}
The second objective of this article is to construct transcendental numbers using known arithmetic functions. A transcendental number is a real or complex number that is not algebraic, that is, not a root of a non-zero polynomial with integer (or, equivalently, rational) coefficients. The well known transcendental numbers are $\pi$ and $e$. All transcendental real numbers are irrational numbers, since all rational numbers are algebraic. The converse is not true, that is, not all irrational numbers are transcendental. The following theorem gives example of infinitely many transcendental numbers generating from the sequences in Theorems \ref{thm1.3}, \ref{thm1.4}, \ref{thm1.5}, \ref{theorem tau}, and \ref{theorem es}.

\begin{theorem}\label{thm1.10}
Under the above notations, let $m$ and $k\geq2$ be two positive integers such that $\alpha_{k,m}$ is a real number whose continued fraction is given by $[0;d_{1,k,m},d_{2,k,m},\ldots]$. Then $\alpha_{k,m}$ is a transcendental number for the following values of $d_{n,k,m}$ and $m$:
\begin{enumerate}
   \item[\upshape(1)] $$d_{n,k,m}=1+\biggl(\left[n\sum_{r\geq1}\frac{\tau(r)\pmod{m}}{10^r}\right] \pmod k\biggl)\quad \text{where~}m=5,7,8,9.$$
       \item[\upshape(2)] $$d_{n,k,m}=1+\biggl(\left[n\sum_{r\geq1}\frac{a(r)\pmod{m}}{10^r}\right] \pmod k\biggl)\quad \text{where~}m=7.$$
        \item[\upshape(3)] $$d_{n,k,m}=1+\biggl(\left[n\sum_{r\geq1}\frac{b(r)\pmod{m}}{10^r}\right] \pmod k\biggl)\quad \text{where~}m=5.$$
         \item[\upshape(4)] $$d_{n,k,m}=1+\biggl(\left[n\sum_{r\geq1}\frac{c(r)\pmod{m}}{10^r}\right] \pmod k\biggl)\quad \text{where~}m=7.$$
          \item[\upshape(5)] $$d_{n,k,m}=1+\biggl(\left[n\sum_{r\geq1}\frac{d(r)\pmod{m}}{10^r}\right] \pmod k\biggl)\quad \text{where~}m=5,7.$$
           \item[\upshape(6)] $$d_{n,k,m}=1+\biggl(\left[n\sum_{r\geq1}\frac{e(r)\pmod{m}}{10^r}\right] \pmod k\biggl)\quad \text{where~}m=5,7.$$
     \item[\upshape(7)] $$d_{n,k,m}=1+\biggl(\left[n\sum_{r\geq1}\frac{\big(\Phi(r)\pmod{m}\big)\pmod{10}}{10^r}\right] \pmod k\biggl)$$ \\ \quad $\text{where~}m\geq7\text{~is an odd integer with~~}m\equiv 7,11,15,19\pmod{20}.$
    \item[\upshape(8)] $$d_{n,k,m}=1+\biggl(\left[n\sum_{r\geq1}\frac{\big(g(r)\pmod{m}\big)\pmod{10}}{10^r}\right] \pmod k\biggl)$$\\ \quad $\text{where }m\geq5 \text{ is an odd integer such that } m\equiv 3,7,9\pmod{10}$, and $\gcd(\varphi(m),5)=1$.
    \item[\upshape(9)] $$d_{n,k,m}=1+\biggl(\left[n\sum_{r\geq1}\frac{\varphi(r)\pmod{m}}{10^r}\right] \pmod k\biggl)\quad \text{where~}3\leq m\leq 10.$$
    \item[\upshape(10)] $$d_{n,k,m}=1+\biggl(\left[n\sum_{r\geq1}\frac{\sigma(r)\pmod{m}}{10^r}\right] \pmod k\biggl) \quad \text{where~} 7\leq m\leq 10.$$
    \item[\upshape(11)] $$d_{n,k,m}=1+\biggl(\left[n\sum_{r\geq1}\frac{\sigma*\varphi(r)\pmod{m}}{10^r}\right] \pmod k\biggl)\quad \text{where~}7\leq m\leq 10.$$
    \item[\upshape(12)] $$d_{n,k,m}=1+\biggl(\left[n\sum_{r\geq1}\frac{J_{v}(r)\pmod{m}}{10^r}\right] \pmod k\biggl)$$\\ \quad where $3\leq m\leq 10$ and $v\geq3$ is an odd integer.
     \item[\upshape(13)] $$d_{n,k,m}=1+\biggl(\left[n\sum_{r\geq1}\frac{\varphi^{*}(r)\pmod{m}}{10^r}\right] \pmod k\biggl)\quad \text{where~}3\leq m\leq 10.$$
    \end{enumerate}
\end{theorem}
Here, and throughout the paper, $[x]$ denotes the integer part of the real number $x$.

\section{Preliminaries}\label{sec2}
In this section, we record several auxiliary results that will be used frequently in the subsequent proofs.\\
The Ramanujan tau-function $\tau(n)$ has many remarkable congruence properties such as:
\begin{theorem}[\cite{newnew, new1,murty}]\label{murty}We have
\noindent
$$\tau(n) \equiv
\begin{cases}
n\,\sigma_{1}(n) \pmod{5},\\[6pt]
n\,\sigma_{3}(n) \pmod{7},\\[6pt]
\sigma_1(n) \pmod{8}~~\text{~~~~if~~}\gcd(n,2)=1,\\[6pt]
n^{2}\sigma_{1}(n) \pmod{9},\\[6pt]
\sigma_{11}(n) \pmod{691}.
\end{cases}$$
\end{theorem}
The Fourier expansions of certain normalized Eisenstein series are given as follows:
\begin{lemma}[\cite{murty}, Chap. 4]\label{lemma es}
We have
$$E_4(z) = 1 + 240 \sum_{n=1}^{\infty} \sigma_3(n) q^n, \quad E_6(z) = 1 - 504 \sum_{n=1}^{\infty} \sigma_5(n) q^n$$
$$E_8(z) = 1 + 480 \sum_{n=1}^{\infty} \sigma_7(n) q^n,\quad E_{10}(z) = 1 - 264 \sum_{n=1}^{\infty} \sigma_9(n) q^n$$
$$E_{14}(z) = 1 - 24 \sum_{n=1}^{\infty} \sigma_{13}(n) q^n.$$
\end{lemma}
The authors together with Mohan \cite{r2} established the following property of Nathanson’s totient function $\Phi(n)$, which will play a pivotal role later:
\begin{lemma}[\cite{r2}]\label{c1}
    If $p$ is a prime and $p \mid n$, then $\Phi(p)\mid \Phi(n)$.
\end{lemma}

Ayad and Kihel \cite{r1} gave the following relation between $g(n)$ and $\Phi(n)$.
\begin{lemma}[\cite{r1}, Lemma 1]\label{lemma1.5}
For any integer $n\geq1$, we have 
$$g(n+1)-g(n)=\frac{1}{2}\Phi(n+1).$$
\end{lemma}

Finally, we recall the following theorem due to Adamczewski, Bugeaud, and Davison \cite{r4}, to construct transcendental numbers from certain irrational numbers, we use this theorem to prove Theorem \ref{thm1.10}.
\begin{theorem}[\cite{r4}, Theorem 4.1]\label{thm*}
Let $\theta$ be an irrational number with $0<\theta<1$ and let $k$ be an integer at least equal to $2$. Let $\textbf{d}=(d_n)_{n\geq1}$ be defined by $d_n=1+([n\theta] \pmod k)$ for any $n\geq1$. Then, the number $\alpha_{k,\theta}=[0;d_1,d_2,\ldots]$ is transcendental.
\end{theorem}

\section{Proof of Main Results}
\begin{proof}[\textbf{Proof of Theorem \ref{thm1.8}}]
Assume, for the sake of contradiction, that the sequence $\{f(n) \pmod{m}\}_{n\geq 1}$ is periodic. Then there exist fixed positive integers $N,L$ such that
\begin{align}\label{A}
f(n+L)\equiv f(n) \pmod{m}\quad \text{for all~}n\geq N.    
\end{align}

In particular, choose $n=p_N\geq N$, the smallest prime such that $f(p_N)\not\equiv 0\pmod{m}$  (existence is guaranteed by the fact that $m\nmid f(p)$ for infinitely many primes $p$), then from (\ref{A}) we get
\begin{align}\label{1}
f(p_N+jL)\equiv f(p_N)\not\equiv 0\pmod{m}\quad \text{for all~}j\in\mathbb{N}.    
\end{align}
Setting $j=p_NK(m)j_N$ in (\ref{1}), where $j_N$ is some positive integer, we obtain
\begin{align}\label{2}
f(p_N(1+K(m)j_NL))\equiv f(p_N)\not\equiv 0\pmod{m}.    
\end{align}
Since $\gcd(1,p_NK(m)L)=1$, by the Dirichlet prime number theorem, the following sequence
$$1, \quad 1+p_NK(m)L, \quad 1+2p_NK(m)L,\quad \ldots,\quad 1+j_N'p_NK(m)L,\quad \ldots$$
contains infinitely many primes. Without loss of generality, assume that $1+j_N'p_NK(m)L$ is the smallest prime in this sequence. Then choosing $j_N=j_N'p_N$ in (\ref{2}) we get
\begin{align}\label{eq3}
f(p_N(1+K(m)j_N'p_NL))\equiv f(p_N)\not\equiv 0\pmod{m}.    
\end{align}
Since $1+K(m)j_N'p_NL$ is a prime and satisfies $1+K(m)j_N'p_NL\equiv 1 \pmod{K(m)}$, we obtain  $m\mid f(1+K(m)j_N'p_NL)$. Moreover, using the fact that $f(p)\mid f(n)$ whenever $p\mid n$ and observing that $1+K(m)j_N'p_NL\mid p_N(1+K(m)j_N'p_NL)$ we deduce  $f(1+K(m)j_N'p_NL)\mid f(p_N(1+K(m)j_N'p_NL))$. Consequently, it follows that
\begin{align}\label{1*}
 0\equiv f(p_N(1+K(m)j_N'p_NL))\equiv f(p_N)\not\equiv 0\pmod{m},   
\end{align}

which is a contradiction. This completes the proof.
\end{proof}
\begin{proof}[\textbf{Proof of Theorem \ref{thm1.9}}]
 If possible assume that $\{f(n) \pmod{m}\}_{n\geq 1}$ is a periodic sequence. Then there exist fixed positive integers $N,L$ such that
 \begin{align}\label{B}
  f(n+L)\equiv f(n) \pmod{m}\quad \text{for all~}n\geq N.   
 \end{align}
In particular, setting $n=p_N\geq N$ in (\ref{B}), the smallest prime such that $f(p_N)\equiv r_2\pmod{m}$ (existence is guaranteed by the fact that $f(p)\equiv r_2\pmod{m}$ for infinitely many primes $p$), we obtain
\begin{align}\label{eq4}
f(p_N+jL)\equiv f(p_N)\equiv r_2\pmod{m}\quad \text{for all~}j\in\mathbb{N}.    
\end{align}
Choosing $j=p_Nmj_N$ in (\ref{eq4}), where $j_N$ is some positive integer, we get
\begin{align}\label{eq5}
f(p_N(1+mj_NL))\equiv f(p_N)\equiv r_2\pmod{m}.    
\end{align}
Since $\gcd(1,p_NmL)=1$, by the Dirichlet prime number theorem, the following sequence
$$1, \quad 1+p_NmL, \quad 1+2p_NmL,\quad \ldots,\quad 1+j_N'p_NmL,\quad \ldots$$
contains infinitely many primes. Without loss of generality, assume that $1+j_N'p_NmL$ is the smallest prime in this sequence. Then setting $j_N=j_N'p_N$ in (\ref{eq5}) we get
\begin{align}\label{eq6}
f(p_N(1+mj_N'p_NL))\equiv f(p_N)\equiv r_2\pmod{m}.    
\end{align}
Since $f$ is multiplicative function for square-free integers $n$, we have $f(p_N(1+mj_N'p_NL))=f(p_N)f(1+mj_N'p_NL)$ and since $1+mj_N'p_NL$ is a prime such that $1+mj_N'p_NL\equiv 1 \pmod{m}$, we get $f(1+mj_N'p_NL)\equiv r_1\pmod{m}$, consequently it follows from (\ref{eq6}) that
$$r_2r_1\equiv f(p_N)f(1+mj_N'p_NL) =f(p_N(1+mj_N'p_NL)\equiv f(p_N)\equiv r_2\pmod{m}$$
which contradicts the fact that $r_2,r_2r_1$ are distinct in $\mathbb{Z}/m\mathbb{Z}$ . This completes the proof.   
\end{proof}
Our proof of Theorem \ref{theorem tau} is based on an application of Theorem \ref{murty}.
\begin{proof}[\textbf{Proof of Theorem \ref{theorem tau}}]
Let $p$ be a prime. We treat each value of $m$ separately.
\noindent
\begin{enumerate}
    \item Let $m=5$. If $p\equiv 1 \pmod{5}$ then $\tau(p)\equiv p(p+1)\equiv 2 \pmod{5}$ and if $p\equiv 2 \pmod{5}$ then $\tau(p)\equiv p(p+1)\equiv 1 \pmod{5}$, therefore by the Dirichlet prime number theorem, we have $5\nmid \tau(p)$ for infinitely many primes $p$. The conclusion follows from Theorem \ref{thm1.9} with $r_1=2,r_2=1$.
     \item Let $m=7$. If $p\equiv 1 \pmod{7}$ then $\tau(p)\equiv p(p^3+1)\equiv 2 \pmod{7}$ and if $p\equiv 2 \pmod{7}$ then $\tau(p)\equiv p(p^3+1)\equiv 4 \pmod{7}$, therefore by the Dirichlet prime number theorem, we have $7\nmid \tau(p)$ for infinitely many primes $p$. Hence, the result follows from Theorem \ref{thm1.9} with $r_1=2,r_2=4$.
      \item Let $m=8$. If $p\equiv 1 \pmod{8}$ then $\tau(p)\equiv p+1\equiv 2 \pmod{8}$ and if $p\equiv 3 \pmod{8}$ then $\tau(p)\equiv p+1\equiv 4 \pmod{8}$, therefore by the Dirichlet prime number theorem, we have $8\nmid \tau(p)$ for infinitely many primes $p$, therefore, the prove now follows from Theorem \ref{thm1.9} with $r_1=2,r_2=4$.
      \item Let $m=9$. If $p\equiv 1 \pmod{9}$ then $\tau(p)\equiv p^2(p+1)\equiv 2 \pmod{9}$ and if $p\equiv 2 \pmod{9}$ then $\tau(p)\equiv p^2(p+1)\equiv 3 \pmod{9}$, therefore by the Dirichlet prime number theorem, we have $9\nmid \tau(p)$ for infinitely many primes $p$, therefore, the prove now follows from Theorem \ref{thm1.9} with $r_1=2,r_2=3$.
    \item Let $m=691$. If $p\equiv 1 \pmod{691}$ then $\tau(p)\equiv p^{11}+1\equiv 2 \pmod{691}$ and if $p\equiv 2 \pmod{691}$ then $\tau(p)\equiv p^{11}+1\equiv 667 \pmod{691}$, therefore by the Dirichlet prime number theorem, we have $691\nmid \tau(p)$ for infinitely many primes $p$. Applying Theorem \ref{thm1.9} with $r_1=2,r_2=667$ completes the proof.
\end{enumerate}
\end{proof}
\begin{remark}
However, it is worth noting that the classical Ramanujan congruences \cite{new1,new2}, $\tau(n)\equiv n^2\sigma_1(n) \pmod{3}$ and $\tau(n)\equiv n^3\sigma_1(n) \pmod{4}$ do not suffice, together with Theorem \ref{thm1.9} to obtain such non-periodicity results for $\{\tau(n) \pmod{3}\}_{n\geq1}$ and $\{\tau(n) \pmod{4}\}_{n\geq1}$.
\end{remark}
The proof of Theorem \ref{theorem es} requires the following lemma.
\begin{lemma}\label{lemmaess}
  If the sequence $\{f(n) \pmod{m}\}_{n\geq1}$ is non-periodic, then the sequence $\{\lambda f(n) \pmod{m}\}_{n\geq1}$ is also non-periodic, where $\gcd(\lambda,m)=1$.  
\end{lemma}
\begin{proof}
We argue by contradiction. Suppose that the sequence $\{\lambda f(n) \pmod{m}\}_{n\geq 1}$ is periodic. Then there exist fixed positive integers $N,L$ such that
\begin{align}
\lambda f(n+L)\equiv \lambda f(n) \pmod{m}\quad \text{for all~}n\geq N.    
\end{align}
Since $\gcd(\lambda ,m)=1$, we get
$$f(n+L)\equiv f(n) \pmod{m}\quad \text{for all~}n\geq N,$$
which implies that the sequence $\{f(n) \pmod{m}\}_{n\geq1}$ is eventually periodic, which is a contradiction. This completes the proof.
\end{proof}

\begin{proof}[\textbf{Proof of Theorem \ref{theorem es}}]
Let $p$ be a prime. 
\noindent
\begin{enumerate}
    \item From Lemma \ref{lemma es} we get $a(n)=240\sigma_3(n)$ for $n\geq1$. We first prove that $\{\sigma_3(n) \pmod{m}\}_{n\geq1}$ is non-periodic then the proof follows from Lemma \ref{lemmaess}. If $p\equiv 1 \pmod{m}$ then $\sigma_3(p)=p^3+1\equiv 2 \pmod{m}$ and if $p\equiv 2 \pmod{m}$ then $\sigma_3(p)=p^3+1\equiv 9 \pmod{m}$, therefore by the Dirichlet prime number theorem, we have $m\nmid \sigma_3(p)$ for infinitely many primes $p$. The conclusion follows from Theorem \ref{thm1.9} with $r_1=2,r_2=9 \pmod{m}$.
    \item From Lemma \ref{lemma es} we get $b(n)=-504\sigma_5(n)$ for $n\geq1$. We first prove that $\{\sigma_5(n) \pmod{m}\}_{n\geq1}$ is non-periodic then the proof follows from Lemma \ref{lemmaess}. If $p\equiv 1 \pmod{m}$ then $\sigma_5(p)=p^5+1\equiv 2 \pmod{m}$ and if $p\equiv 2 \pmod{m}$ then $\sigma_5(p)=p^5+1\equiv 33 \pmod{m}$, therefore by the Dirichlet prime number theorem, we have $m\nmid \sigma_5(p)$ for infinitely many primes $p$. Hence, the result follows from Theorem \ref{thm1.9} with $r_1=2 \pmod{m},r_2=33 \pmod{m}$.
     \item From Lemma \ref{lemma es} we get $c(n)=480\sigma_7(n)$ for $n\geq1$. We first prove that $\{\sigma_7(n) \pmod{m}\}_{n\geq1}$ is non-periodic then the proof follows from Lemma \ref{lemmaess}. If $p\equiv 1 \pmod{m}$ then $\sigma_7(p)=p^7+1\equiv 2 \pmod{m}$ and if $p\equiv 2 \pmod{m}$ then $\sigma_7(p)=p^7+1\equiv 129 \pmod{m}$, therefore by the Dirichlet prime number theorem, we have $m\nmid \sigma_7(p)$ for infinitely many primes $p$. Applying Theorem \ref{thm1.9} with $r_1=2,r_2=129\pmod{m}$ completes the proof.
      \item From Lemma \ref{lemma es} we get $d(n)=-264\sigma_9(n)$ for $n\geq1$. We first prove that $\{\sigma_9(n) \pmod{m}\}_{n\geq1}$ is non-periodic then the proof follows from Lemma \ref{lemmaess}. If $p\equiv 1 \pmod{m}$ then $\sigma_9(p)=p^9+1\equiv 2 \pmod{m}$ and if $p\equiv 2 \pmod{m}$ then $\sigma_9(p)=p^9+1\equiv 513\pmod{m}$, therefore by the Dirichlet prime number theorem, we have $m\nmid \sigma_9(p)$ for infinitely many primes $p$. The conclusion follows from Theorem \ref{thm1.9} with $r_1=2,r_2=513\pmod{m}$.
       \item From Lemma \ref{lemma es} we get $e(n)=-24\sigma_{13}(n)$ for $n\geq1$. We first prove that $\{\sigma_{13}(n) \pmod{m}\}_{n\geq1}$ is non-periodic then the proof follows from Lemma \ref{lemmaess}. If $p\equiv 1 \pmod{m}$ then $\sigma_{13}(p)=p^{13}+1\equiv 2 \pmod{m}$ and if $p\equiv 2 \pmod{m}$ then $\sigma_{13}(p)=p^{13}+1\equiv 8193\pmod{m}$, therefore by the Dirichlet prime number theorem, we have $m\nmid \sigma_{13}(p)$ for infinitely many primes $p$. Hence, the result follows from Theorem \ref{thm1.9} with $r_1=2,r_2=8193\pmod{m}$.
    \end{enumerate}
\end{proof}

\begin{proof}[\textbf{Proof of Theorem \ref{thm1.4}}]
Let $p$ be a prime. If $p\equiv 1\pmod{\varphi(m)}$, then by Euler's totient theorem, we have
$$\Phi(p)=2^{p}-2\pmod{m}=0.$$
If $p\equiv\varphi(m)-1\pmod{\varphi(m)}$ then
$$\Phi(p)=2^{p}-2\pmod{m}=\frac{m-3}{2}\neq 0\quad \text{as~}m\geq5.$$
Since $\gcd(\varphi(m)-1,\varphi(m))=1$, thanks to the Dirichlet prime number theorem for ensuring that $m\nmid \Phi(p)$ for infinitely many primes $p$. Now the proof follows from Lemma \ref{c1} and Theorem \ref{thm1.8} with $K(m)=\varphi(m)$.
\end{proof}

\begin{proof}[\textbf{Proof of Theorem \ref{thm1.3}}]
Suppose, for the sake of contradiction, that the sequence $\{g(n) \pmod{m}\}_{n\geq1}$ is a periodic sequence for all odd positive integers $m\geq 5$. Then Lemma \ref{lemma1.5} implies that $\{\Phi(n) \pmod{m}\}_{n\geq 1}$ is an eventually periodic sequence for all odd positive integers $m\geq 5$, which immediately contradicts Theorem \ref{thm1.4}. This completes the proof.
\end{proof}

\begin{proof}[\textbf{Proof of Theorem \ref{thm1.5}}]\noindent
\begin{enumerate}
    \item If $p$ is a prime such that $p\equiv 1\pmod{m}$ then $m\mid \varphi(p)$, and if $p\equiv m-1\pmod{m}$ then $\varphi(p)\equiv m-2\pmod{m}$, since $m\geq3$ and $\gcd(m-1,m)=1$, by the Dirichlet prime number theorem, it follows that $m\nmid\varphi(p)$ for infinitely many primes $p$. Hence, the result follows from Theorem \ref{thm1.8} with $K(m)=m$.

Also, we can set $r_1=0,r_2=m-2$ in Theorem \ref{thm1.9} to conclude the proof.

\item For any prime $p\equiv 1\pmod{m}$ we have $\sigma(p)\equiv2\pmod{m}$, and for any prime $p\equiv r \pmod{m}$ where $2\leq r\leq m-2$ such that $\gcd(r,m)=1$ (existence of $r$ follows from the fact that $\varphi(m)>2$), we have $\sigma(p)\equiv r+1\pmod{m}$. Since $\gcd(r,m)=1$, by the Dirichlet prime number theorem, we have infinitely many primes $p$ satisfying $\sigma(p)\equiv r+1\pmod{m}$. Applying Theorem \ref{thm1.9} with $r_1=2,r_2=r+1$ completes the proof.
\item If $p$ is a prime such that $p\equiv 1\pmod{m}$ then $\sigma*\varphi(p) \equiv2\pmod{m}$, and if $p\equiv r \pmod{m}$ where $2\leq r\leq m-2$ such that $(r,m)=1$ (existence of $r$ follows from the fact that $\varphi(m)>2$), we have $\sigma*\varphi(p)\equiv 2r\pmod{m}$. Since $\gcd(r,m)=1$, by the Dirichlet prime number theorem, we have infinitely many primes $p$ satisfying $\sigma*\varphi(p)\equiv 2r\pmod{m}$. The conclusion follows from Theorem \ref{thm1.9} with $r_1=2,r_2=2r$.

\item If $p$ is a prime such that $p\equiv 1\pmod{m}$ then $m\mid J_k(p)$, and if $p\equiv m-1\pmod{m}$ then $J_k(p)\equiv (m-1)^k-1\equiv m-2 \pmod{m}$, since $m\geq3$ and $\gcd(m-1,m)=1$, by the Dirichlet prime number theorem, it follows that $m\nmid J_k(p)$ for infinitely many primes $p$. Hence, the result follows from Theorem \ref{thm1.8} with $K(m)=m$.
Also, we can set $r_1=0,r_2=m-2$ in Theorem \ref{thm1.9} to conclude the proof.
\item For any prime $p\equiv 1\pmod{m}$ we have $m\mid \varphi^{*}(p)$, and if $p\equiv m-1\pmod{m}$ then $\varphi^{*}(p)\equiv m-2 \pmod{m}$, since $m\geq3$ and $\gcd(m-1,m)=1$, by the Dirichlet prime number theorem, it follows that $m\nmid \varphi^{*}(p)$ for infinitely many primes $p$. Hence, the result follows from Theorem \ref{thm1.8} with $K(m)=m$.
Also, we can set $r_1=0,r_2=m-2$ in Theorem \ref{thm1.9} to conclude the proof.
\end{enumerate}
\end{proof}

To prove Theorem \ref{thm1.10}, we need the following lemmas.
\begin{lemma}\label{lemma3.1}
$\{\big(\lambda \Phi(n) \pmod{m}\big)\pmod{10}\}_{n\geq 1}$ is a non-periodic sequence for all odd positive integers $m\geq 5$ with $m\equiv 7,11,15,19\pmod{20}$, and $\lambda =\frac{1}{2},1$.
\end{lemma}
\begin{proof}
Assume, for the sake of contradiction, that the sequence $\{\big(\lambda \Phi(n) \pmod{m}\big)\pmod{10}\}_{n\geq 1}$ is periodic. Then there exist fixed positive integers $N,L$ such that
\begin{align}\label{10}
\lambda \Phi(n+L) \pmod{m}\equiv \lambda \Phi(n) \pmod{m} \pmod{10}\quad \text{for all~}n\geq N.    
\end{align}
Since there are infinitely many primes of the form $p\equiv 1\pmod{\varphi(m)}$ and $p\equiv\varphi(m)-1\pmod{\varphi(m)}$, we let $p_N\equiv \varphi(m)-1\pmod{\varphi(m)}$ be the smallest prime greater than $N$. Then setting $n=p_N$ in (\ref{10}) we get
\begin{align}\label{11}
\lambda \Phi(p_N+jL) \pmod{m}\equiv \lambda \Phi(p_N) \pmod{m} \pmod{10}\quad \text{for all~}j\in\mathbb{N}.    
\end{align}
Setting $j=p_N\varphi(m)j_N$ in (\ref{11}), where $j_N$ is some positive integer, we obtain
\begin{align}\label{12}
\lambda \Phi(p_N(1+\varphi(m)j_NL))\pmod{m}\equiv \lambda \Phi(p_N) \pmod{m}\pmod{10}.    
\end{align}
Since $\gcd(1,p_N\varphi(m)L)=1$, by the Dirichlet prime number theorem, the following sequence
$$1, \quad 1+p_N\varphi(m)L, \quad 1+2p_N\varphi(m)L,\quad \ldots,\quad 1+j_N'p_N\varphi(m)L,\quad \ldots$$
contains infinitely many primes. Without loss of generality, assume that $1+j_N'p_N\varphi(m)L$ is the smallest prime in this sequence. Then choosing $j_N=j_N'p_N$ in (\ref{12}) we get
\begin{align}\label{13}
\lambda \Phi(p_N(1+\varphi(m)j_N'p_NL))\pmod{m}\equiv \lambda \Phi(p_N) \pmod{m}\pmod{10}.    
\end{align}
Since $1+\varphi(m)j_N'p_NL$ is a prime and satisfies $1+\varphi(m)j_N'p_NL\equiv 1 \pmod{\varphi(m)}$, and  since $\gcd(\frac{1}{\lambda},m)=1$ and $\Phi(n)$ is even for $n\geq2$, we obtain  $\frac{m}{\lambda}\mid \Phi(1+\varphi(m)j_N'p_NL)$. Moreover, using Lemma \ref{c1} and observing that $1+\varphi(m)j_N'p_NL\mid p_N(1+\varphi(m)j_N'p_NL)$, we deduce  $\Phi(1+\varphi(m)j_N'p_NL)\mid \Phi(p_N(1+\varphi(m)j_N'p_NL))$, consequently, it follows that
$$\lambda\Phi(p_N(1+\varphi(m)j_N'p_NL))\equiv 0\pmod{m}.$$
Since $p_N\equiv \varphi(m)-1\pmod{\varphi(m)}$ is a prime and $m\equiv 7,11,15,19\pmod{20}$, we have
$$\lambda\Phi(p_N)=\lambda(2^{p_N}-2)\equiv \lambda\frac{m-3}{2}\pmod{m},\quad\text{and}\quad  \lambda\frac{m-3}{2} \not\equiv 0 \pmod{5},$$
it follows from (\ref{13}) that
\begin{align*}
0\equiv\lambda\Phi(p_N(1+\varphi(m)j_N'p_NL))\pmod{m}\equiv \lambda\Phi(p_N) \pmod{m}=\lambda\frac{m-3}{2}\not\equiv 0\pmod{10}    
\end{align*}
which is absurd. This completes the proof.
\end{proof}

\begin{lemma}\label{Lemma new}
Let $m\geq 5$ be an integer and define
$$T_m(x)=x \pmod{m}\pmod{10}.$$
Suppose that, for some integers $x,y$, $T_m(x+y)=T_m(x)$. Then
$$(y\pmod m)\pmod {10}\in\{0,m\pmod {10}\}.$$
\end{lemma}
\begin{proof}
Set
$$r=x\pmod m,\qquad s=y\pmod m,$$
where $0\leq r,s<m$. Since $x\equiv r\pmod m,~y\equiv s\pmod m,$
we have
$$(x+y)\pmod m\equiv (r+s)\pmod m.$$ 
Consider the following two cases.\\
Case 1: If $r+s<m$. Then $(x+y)\pmod m=r+s$. Hence, from $T_m(x+y)=T_m(x)$, we have
$$(r+s)\pmod {10}=r\pmod {10}.$$
Therefore $s\equiv0\pmod{10},$ and consequently $s\pmod {10}=0$.\\
Case 2: If $r+s\geq m$, then clearly $r+s<2m$ and 
$$(x+y)\pmod m=r+s-m.$$
Thus
$(r+s-m)\pmod {10} =r\pmod {10},$
which gives $s\equiv m\pmod{10}$. Therefore
$$s\pmod {10}=m \pmod {10}.$$
Since $s\equiv y\pmod m$, in both cases, we obtain
$$(y\pmod m)\pmod {10}\in\{0,m\pmod {10}\}.$$
This proves the lemma.
\end{proof}

\begin{lemma}\label{lemma3.2}
$\{\big(g(n) \pmod{m}\big)\pmod{10}\}_{n\geq 1}$ is a non-periodic sequence for all odd positive integers $m\geq5$ such that $m\equiv 3,7,9\pmod{10}$ and $\gcd(\varphi(m),5)=1$.
\end{lemma}
\begin{proof}
Assume, for the sake of contradiction, that the sequence $\{\big(g(n) \pmod{m}\big)\pmod{10}\}_{n\geq 1}$ is periodic. Then there exist fixed positive integers $N,L$ such that
\begin{align}\label{14}
g(n+L) \pmod{m}\equiv g(n) \pmod{m} \pmod{10}\quad \text{for all~}n\geq N,    
\end{align}
then for all $n\geq N+1$ and by Lemma \ref{lemma1.5} we have
\begin{align}\label{15}
\frac{1}{2}\Phi(n+L)+g(n-1+L) \pmod{m}\equiv \frac{1}{2}\Phi(n)+g(n-1) \pmod{m} \pmod{10}.    
\end{align}
Choose a prime  $r$ such that $r\equiv 1 \pmod{L\varphi(m)}$, i.e, $r=L\varphi(m)j+1$, where $j$ is some positive integer. Then $m\mid \Phi(r)$. Since $\gcd(5,\varphi(m))=1$, we have infinitely many primes $q$ such that $q\equiv 5 \pmod{\varphi(m)}$, therefore, choose a suitable prime $q\geq N+1$.  Then from (\ref{15}), we have
$$\frac{1}{2}\Phi(q+q\varphi(m)jL)+g(q-1+q\varphi(m)jL) \pmod{m}\equiv \frac{1}{2}\Phi(q)+g(q-1) \pmod{m} \pmod{10}.$$
Since $q+q\varphi(m)jL=q(1+\varphi(m)jL)=qr$, we have by Lemma \ref{c1} that $\Phi(r)\mid \Phi(q+q\varphi(m)jL)$, it follows that $m\mid \Phi(q+q\varphi(m)jL)$ as $m\mid \Phi(r)$. Therefore, using (\ref{14}), we get
\begin{align*}
  g(q-1) \pmod{m}\equiv g(q-1+jq\varphi(m)L) \pmod{m}\equiv \frac{1}{2}\Phi(q)+g(q-1) \pmod{m}  \pmod{10}.  
\end{align*}
Set $x=g(q-1)$, $y=\frac{1}{2}\Phi(q)$ in Lemma \ref{Lemma new}, then we have
\begin{align}\label{17}
\frac{1}{2}\Phi(q)\pmod{m}\pmod{10}\in\{0,m\pmod {10}\}.    \end{align}
Since $q\equiv 5\pmod{\varphi(m)}$, we have
$$\frac{1}{2}\Phi(q)\equiv 15\pmod{m},$$
consequently, we have
$$\frac{1}{2}\Phi(q)\pmod{m}\pmod{10}=\begin{cases}
    1,\quad \text{if } m=7,\\
    6,\quad \text{if } m=9,\\
    2,\quad \text{if } m=13,\\
    5,\quad \text{if } m\geq 17,\\
\end{cases}$$
which contradicts (\ref{17}).\\
Hence, $\{\big(g(n) \pmod{m}\big)\pmod{10}\}_{n\geq 1}$ is a non-periodic sequence for all odd positive integers $m\geq5$ such that $m\equiv 3,7,9\pmod{10}$ and $\gcd(\varphi(m),5)=1$.
This completes the proof.
\end{proof}

\begin{proof}[\textbf{Proof of Theorem \ref{thm1.10}}]

\noindent
\begin{enumerate}
 \item The proof follows from the fact that  
$$\sum_{r\geq1}\frac{\tau(r)\pmod{m}}{10^r}$$
is an irrational number by Theorem \ref{theorem tau} and from Theorem \ref{thm*}.
 \item The proof follows from the fact that  
$$\sum_{r\geq1}\frac{a(r)\pmod{m}}{10^r}$$
is an irrational number by Theorem \ref{theorem es} and from Theorem \ref{thm*}.
\item The proof follows from the fact that  
$$\sum_{r\geq1}\frac{b(r)\pmod{m}}{10^r}$$
is an irrational number by Theorem \ref{theorem es} and from Theorem \ref{thm*}.
\item The proof follows from the fact that  
$$\sum_{r\geq1}\frac{c(r)\pmod{m}}{10^r}$$
is an irrational number by Theorem \ref{theorem es} and from Theorem \ref{thm*}.
\item The proof follows from the fact that  
$$\sum_{r\geq1}\frac{d(r)\pmod{m}}{10^r}$$
is an irrational number by Theorem \ref{theorem es} and from Theorem \ref{thm*}.
\item The proof follows from the fact that  
$$\sum_{r\geq1}\frac{e(r)\pmod{m}}{10^r}$$
is an irrational number by Theorem \ref{theorem es} and from Theorem \ref{thm*}.
 \item  The proof follows from the fact that  
 $$\sum_{r\geq1}\frac{\big(\Phi(r)\pmod{m}\big)\pmod{10}}{10^r}$$
 is an irrational number by Lemma \ref{lemma3.1} and from Theorem \ref{thm*}. 
 \item  The proof follows from the fact that  
 $$\sum_{r\geq1}\frac{\big(g(r)\pmod{m}\big)\pmod{10}}{10^r}$$
 is an irrational number by Lemma \ref{lemma3.2} and from Theorem \ref{thm*}. 
 \item The proof follows from the fact that  
$$\sum_{r\geq1}\frac{\varphi(r)\pmod{m}}{10^r}$$
is an irrational number by Theorem \ref{thm1.5} and from Theorem \ref{thm*}.
 \item The proof follows from the fact that  
$$\sum_{r\geq1}\frac{\sigma(r)\pmod{m}}{10^r}$$
is an irrational number by Theorem \ref{thm1.5} and from Theorem \ref{thm*}.
 \item The proof follows from the fact that  
$$\sum_{r\geq1}\frac{\sigma*\varphi(r)\pmod{m}}{10^r}$$
is an irrational number by Theorem \ref{thm1.5} and from Theorem \ref{thm*}.
 \item The proof follows from the fact that  
$$\sum_{r\geq1}\frac{J_v(r)\pmod{m}}{10^r}$$
is an irrational number by Theorem \ref{thm1.5} and from Theorem \ref{thm*}.
 \item The proof follows from the fact that  
$$\sum_{r\geq1}\frac{\varphi^{*}(r)\pmod{m}}{10^r}$$
is an irrational number by Theorem \ref{thm1.5} and from Theorem \ref{thm*}.

\end{enumerate}
\end{proof}

\section*{Acknowledgment}
The authors are grateful to the anonymous referee for the careful reading of the manuscript and for the valuable suggestions that greatly improved the clarity and presentation of the paper. 
\section{Data Availability} 	
The authors confirm that their manuscript has no associated data.

\section{Competing Interests}
The authors confirm that they have no competing interest.

\end{document}